\documentclass{amsart}
\usepackage[all]{xy}
\usepackage{enumerate}
\usepackage{mathrsfs}
\usepackage{amssymb}
\usepackage{graphicx}
\usepackage{tikz-cd}

\theoremstyle{plain}
\newtheorem{theorem}{Theorem}
\newtheorem{lemma}[theorem]{Lemma}

\newtheorem{proposition}[theorem]{Proposition}

\theoremstyle{definition}
\newtheorem{definition}[theorem]{Definition}

\theoremstyle{remark}
\newtheorem{remark}[theorem]{Remark}

\newcommand\Aut{\operatorname{Aut}}

\newcommand\s{\mathbf{s}}
\newcommand\R{\mathbb{R}}
\newcommand\C{\mathbb{C}}
\newcommand\Sp{\operatorname{Sp}}
\newcommand\GL{\operatorname{GL}}
\newcommand\calO{\mathcal{O}}
\newcommand\Ad{\operatorname{Ad}}

\title[Generalized Gelfand pairs attached to $3$-step nilpotent Lie groups]{A family of generalized Gelfand pairs\\ attached to $3$-step nilpotent Lie groups}

\author{Mai Katada}
\author{Cornelie Mitcha Malanda}

\address{Faculty of Mathematics, Kyushu University, 744, Motooka, Nishi-ku, Fukuoka, Japan 819-0395}
\address{\textit{Current address}: Graduate School of Mathematical Sciences, University of Tokyo, Tokyo 153-8914, Japan}
\email{mkatada@ms.u-tokyo.ac.jp}
\address{Faculty of Sciences and Technology, Marien Ngouabi University, Congo}
\address{\textit{Current address}: ZEN University, 3-12-11 Shinjuku, Zushi-shi, Kanagawa, Japan}
\email{Cornelie.mitcha@yahoo.fr}



\begin{document}

\begin{abstract}
It is well known that if $(K\ltimes N, K)$ is a Gelfand pair with $N$ a nilpotent Lie group and $K$ a compact subgroup of the automorphism group of $N$, then $N$ is at most $2$-step.
The notion of Gelfand pairs is extended to the notion of generalized Gelfand pairs, where $K$ is not necessarily compact.
Gallo and Saal constructed an example of a generalized Gelfand pair $(K\ltimes N, K)$ with $N$ $3$-step and $K$ non-compact.
In this work, we construct a family of generalized Gelfand pairs $(K_d \ltimes N_d, K_d)_{d\ge 1}$, where $N_d$ is $3$-step and $K_d$ is isomorphic to $\R^{d+1}$; the case of $d=1$ recovers the example of Gallo and Saal.
\end{abstract}
\maketitle

\section{Introduction}
Let $G$ be a unimodular Lie group and $K$ a compact subgroup. We say that $(G,K)$ is a Gelfand pair if the convolution algebra of $K$-bi-invariant functions on $G$ is commutative, or equivalently, if every irreducible representation of $K$ occurs at most once in the regular representation of $G$ on $L^2(G/K)$. 
Gelfand pairs play a central role in non-commutative harmonic analysis, providing a natural framework for extending harmonic analytical methods to the study of commutative Banach algebras.
Gelfand pairs are also crucial in the representation theory of non-compact semisimple Lie groups \cite{knapp2001representation, harish1958spherical, helgason1984groups}.

There is a substantial body of literature on Gelfand pairs \cite{faraut1980analyse, helgason1984groups, van2009introduction}.
In particular, a great deal of attention has been paid to \emph{nilpotent Gelfand pairs} $(K\ltimes N,K)$ (often denoted by $(K,N)$), where $N$ is a nilpotent Lie group and $K$ is a compact subgroup of the automorphism group of $N$ \cite{benson1990gel, benson1992bounded, benson1999orbit, lauret2000gelfand, fischer2012nilpotent}.
In \cite{benson1990gel}, Benson, Jenkins, and Ratcliff showed that if $(K,N)$ is a nilpotent Gelfand pair, then $N$ is at most $2$-step. In particular, they studied Gelfand pairs associated with the $(2d+1)$-dimensional Heisenberg group $H_d$, which is a $2$-step nilpotent Lie group.

In the 1980s, the notion of Gelfand pairs was extended to that of generalized Gelfand pairs, in which $K$ is allowed to be non-compact. The theory of generalized Gelfand pairs was substantially developed and popularized through the work of Gerrit van Dijk \cite{van1984generalized, van1986class, van2008gelfand, van2009introduction}. 

The study of generalized Gelfand pairs associated with nilpotent Lie groups has recently been advanced \cite{gallo2020generalized, campos2023generalized, campos2024spherical}.
In \cite{gallo2020generalized}, Gallo and Saal constructed the first example of a generalized Gelfand pair associated with a 3-step nilpotent Lie group.
Subsequently in \cite{campos2023generalized}, Campos, Garc\'{i}a, and Saal constructed a family of generalized Gelfand pairs associated with $m$-step nilpotent Lie groups with $m\ge 3$, where the case of $m=3$ recovers the example of Gallo and Saal.
In \cite{campos2024spherical}, they provided another family of generalized Gelfand pairs associated with $m$-step nilpotent Lie groups with $m\ge 3$.
In the present paper, we construct a new family of generalized Gelfand pairs associated with 3-step nilpotent Lie groups. Our main result is the following.

\begin{theorem}[Cf. Theorem \ref{maintheorem}]
    We have a family $\{(K_{d},N_{d})\}_{d\ge 1}$ of generalized Gelfand pairs $(K_d,N_d)$, where $N_{d}$ is a $(2d+2)$-dimensional $3$-step nilpotent Lie group and $K_{d}$ is a non-compact subgroup of the automorphism group of $N_{d}$ which is isomorphic to $\R^{d+1}$.
\end{theorem}
If $d=1$, then the pair $(K_1,N_1)$ recovers the generalized Gelfand pair constructed by Gallo and Saal in \cite{gallo2020generalized}. 

In \cite{campos2024spherical}, the spherical analysis associated with the family of generalized Gelfand pairs is developed in detail. It would be interesting to develop the spherical analysis for the family of generalized Gelfand pairs that we construct in the present paper.

\section{A criterion for generalized Gelfand pairs\\ introduced by Gallo and Saal}

Here we briefly recall the definition of generalized Gelfand pairs and the criterion for generalized Gelfand pairs introduced by Gallo and Saal in \cite{gallo2020generalized}.

We first recall the definition of generalized Gelfand pairs following the setting of \cite{van2009introduction}.
Let $G$ be a unimodular Lie group and $K$ a closed unimodular subgroup of $G$.
Let $(\pi,\mathcal{H})$ be an irreducible unitary representation of $G$, where $\mathcal{H}$ is a Hilbert space. Let $\mathcal{H}^{\infty}$ denote the subspace of $\mathcal{H}$ consisting of $C^{\infty}$-vectors, which is defined by $$\mathcal{H}^{\infty}=\{v \in \mathcal{H}\mid 
\pi(-)(v):G\to \mathcal{H} \text{ is }C^{\infty}\},$$
and which is a Fr\'{e}chet space equipped with a natural Sobolev topology.
Let $\mathcal{H}^{-\infty}$ be  the \emph{anti-dual} of $\mathcal{H}^{\infty}$, which consists of \emph{distribution vectors}, which are continuous, anti-linear functionals on $\mathcal{H}^{\infty}$.
The space $\mathcal{H}^{-\infty}$ is endowed with the strong topology. Then we have natural continuous inclusions $$\mathcal{H}^{\infty} \subset \mathcal{H} \subset \mathcal{H}^{-\infty}.$$
The representation $\pi$ of $G$ on $\mathcal{H}$ induces a representation $\pi^\infty$ of $G$ on $\mathcal{H}^{\infty}$ and a representation $\pi^{-\infty}$ of $G$ on $\mathcal{H}^{-\infty}$. 
Let $\mathcal{H}^{-\infty}_K$ be the space of distribution vectors fixed by $K$,
that is, the subspace of $\mathcal{H}^{-\infty}$ defined by
$$\mathcal{H}^{-\infty}_K=\{f \in \mathcal{H}^{-\infty}\mid 
\pi^{-\infty}(k)(f)= f, \;k\in K\}.$$

\begin{definition}
	The pair $(G,K)$ is called a generalized Gelfand pair if for any irreducible unitary representation $(\pi,\mathcal{H})$ of G, the space $\mathcal{H}^{-\infty}_{K} $ is at most one dimensional.  
\end{definition}

In what follows, we restrict to the setting of Gallo and Saal: that is, we consider the pair $(K\ltimes N,K)$, where $N$ is a ($3$-step) nilpotent Lie group and $K$ is a subgroup of the automorphism group $\Aut(N)$ of $N$.
Let $\hat{N}$ (resp. $\hat{K}$) denote the set of equivalence classes of irreducible unitary representations of $N$ (resp. $K$).
Suppose that $K$ satisfies the following condition:
for any $(\rho,\mathcal{H}) \in \hat{N}$, there exists a unitary representation $(\omega_{\rho},\mathcal{H})$ of $K$ such that for any $n\in N$, $k\in K$, $$\rho(k\cdot n)=\omega_{\rho}(k)\rho(n)\omega_{\rho}(k^{-1}).$$
Note here that the representation $\omega_{\rho}$ is known as the intertwining representation (or the metaplectic representation) of $\rho$.

\begin{proposition}[Cf. {\cite[Theorem (C)]{gallo2020generalized}}]\label{GalloSaalcriterion}
In the above setting, the pair $(K\ltimes N,K)$ is a generalized Gelfand pair if and only if for any $\rho\in \hat{N}$, the unitary representation $\omega_\rho$ of $K$ is multiplicity free.
\end{proposition}

\begin{proof}
It follows from Mackey theory that
the representations
\begin{gather*}
    \pi_{\sigma,\rho}(k,n)=\sigma(k)\otimes \omega_\rho(k)\rho(n)
\end{gather*}
with $\sigma\in \hat{K}$ and $\rho\in \hat{N}$
form the complete set of irreducible unitary representations of $K\ltimes N$.
Therefore, by definition, $(K\ltimes N,K)$ is a generalized Gelfand pair if and only if for any $\pi_{\sigma,\rho}$,
the space of distribution vectors fixed by $K$ is at most one dimensional.
Note that the representation $\pi_{\sigma,\rho}$ of $K\ltimes N$ has a distribution vector fixed by $K$ if and only if the representation $\sigma\otimes \omega_\rho$ of $K$ has a distribution vector fixed by $K$.
By the result of Mokni and Thomas 
(see \cite[Theorem (B)]{gallo2020generalized}),
the representation $\sigma\otimes \omega_\rho$ has a distribution vector fixed by $K$ if and only if the dual representation $\sigma^*$ of $K$ appears in the decomposition of $\omega_\rho$ into irreducibles.
Therefore, the statement follows.
\end{proof}

\section{Construction of families of generalized Gelfand pairs}

Here we generalize the results of Gallo and Saal \cite{gallo2020generalized} to construct a new family $\{(K_{d}\ltimes N_{d},K_{d})\}_{d\ge 1}$ of generalized Gelfand pairs associated to $3$-step nilpotent Lie groups $N_{d}$. The generalized Gelfand pair constructed in \cite{gallo2020generalized} corresponds to the case of $d=1$ in our family.

\subsection{A family $\{N_{d}\}_{d\ge 1}$ of $3$-step nilpotent Lie groups}\label{sec:family}
Here, we construct a family $\{N_{d}\}_{d\ge 1}$ of $3$-step nilpotent Lie groups $N_d$, which can be deduced from the work of Ratcliff in \cite{ratcliff1985symbols}.

Let $d\ge 1$.
Let $H_{d}$ be the $(2d+1)$-dimensional Heisenberg group.
An element of $H_{d}$ is written as 
$(x,y,t)=(x_{1},\dots,x_{d},y_{1},\dots,y_{d},t)\in \R^{2d+1}$, and the multiplication of elements of $H_{d}$ is given by
\begin{gather*}
    (x,y,t)\cdot (x',y',t')=\left(x+x',y+y',t+t'+\frac{x y'-y x'}{2}\right),
\end{gather*}
where $x y=\sum_{i=1}^{d}x_i y_i$ for $x=(x_1,\dots,x_d),y=(y_1,\dots,y_d)\in \R^d$.

Let $S$ be a subgroup of $\Sp(2d,\R)$ defined by
$$S=\left\{\s=\begin{pmatrix}
I_d &     0\\ 
s J_d &      I_d
\end{pmatrix} \middle \vert s\in \R\right\}\cong \R,$$
where $I_d$ is the identity matrix and $J_d=(\delta_{i,d-j+1})_{1\le i,j\le d}$ is an anti-diagonal matrix.

Define an action of $S$ on $H_{d}$ by
\begin{gather*}
    \s\cdot (x,y,t)=(x,y+sx_{\uparrow},t),
\end{gather*}
where $x_{\uparrow}=J_d x=(x_{d},x_{d-1},\dots,x_{1})$.
This action of $S$ on $H_{d}$ induces the semi-direct product $N_{d}=S\ltimes H_{d}$, where an element of $N_d$ is written as $(s,x,y,t)\in \R^{2d+2}$, the multiplication is given by
\begin{gather*}
    (s,x,y,t)(s',x',y',t')=\left(s+s',x+x',y+y'+sx'_{\uparrow},t+t'+\frac{xy'-yx'+sxx'_{\uparrow}}{2}\right),
\end{gather*}
the unit element is $(0,0,0,0)$ and the inverse element of $(s,x,y,t)$ is given by
\begin{gather*}
    (s,x,y,t)^{-1}=(-s,-x,-y+sx_{\uparrow},-t).
\end{gather*}

Let $\mathfrak{h}_d$ be the Heisenberg Lie algebra corresponding to $H_d$.
We write an element of $\mathfrak{h}_d$ as $(X,Y,T)$ with $X,Y\in \R^d, T\in \R$. Then the Lie bracket in $\mathfrak{h}_d$ is given by
\begin{gather*}
    [(X,Y,T),(X',Y',T')]=\left(0,0,X Y'-Y X'\right).
\end{gather*}
Let $\mathfrak{s}$ denote the $1$-dimensional Lie algebra associated to the Lie group $S$.
Then the Lie algebra $\mathfrak{n}_d$ associated to the Lie group $N_d$ is isomorphic to the semi-direct product $\mathfrak{s}\ltimes \mathfrak{h}_d$ of $\mathfrak{s}$ and $\mathfrak{h}_d$.
An element of $\mathfrak{n}_d$ is written as $(S,X,Y,T)$ with $S\in \mathfrak{s}, (X,Y,T)\in \mathfrak{h}_d$.
Then the Lie bracket of $\mathfrak{n}_d$ is given by
\begin{gather*}
    [(S,X,Y,T),(S',X',Y',T')]=\left(0,0,SX'_{\uparrow}-S'X_{\uparrow},XY'-YX'\right),
\end{gather*}
where $X_{\uparrow}=J_d X$.
It is easy to see that the center $\mathfrak{c}_d$ of $\mathfrak{n}_d$ is $$\mathfrak{c}_d=\{(0,0,0,T)\mid T\in \R\}\cong \R$$
and that the Lie algebra $\mathfrak{n}_d$ is a $3$-step nilpotent Lie algebra since we have
\begin{gather*}
    [(0,0,Y,T),(S',X',Y',T')]=\left(0,0,0,-YX'\right).
\end{gather*}

\subsection{The family of generalized Gelfand pairs}

Here we construct a new family $\{(K_{d}\ltimes N_{d},K_{d})\}_{d\ge 1}$ of generalized Gelfand pairs.

Let $d\ge 1$.
Define a subgroup $K_d$ of $\GL(2d+2,\R)$ by
\begin{gather*}
K_d=\left\{\mathbf{k}=\begin{pmatrix}
1&0& 0&0\\
0&I_d& 0&0\\
0&k_0 J_d & I_d&0\\
0&k&0&1
\end{pmatrix}
\middle \vert 
\substack{k_0,k_1,\dots,k_d\in \R\\ k=(k_1,\dots,k_d)}
\right\}.
\end{gather*}
Then $K_d$ is isomorphic to $\R^{d+1}$ as abelian groups, and we identify $\mathbf{k}\in K_d$ with $(k_0,k_1,\cdots, k_d)\in \R^{d+1}$.

\begin{lemma}
For $d\ge 1$, $K_{d}$ is a subgroup of $\Aut(N_{d})$.
\end{lemma}

\begin{proof}
 We have an action 
 \begin{gather*}
     \rho_d: K_d\to \Aut(N_d)
 \end{gather*}
 of $K_d$ on $N_d$ defined by
 \begin{gather*}
     \rho_d(\mathbf{k})((s,x,y,t))=(s,x,y+k_0 x_{\uparrow},t+k x)
 \end{gather*}
 for $\mathbf{k}\in K_d$ and $(s,x,y,t)\in N_d$.
 One can check that $\rho_d$ is an action since we have
 \begin{gather*}
 \begin{split}
     &\rho_d(\mathbf{k})((s,x,y,t)(s',x',y',t'))\\
     &=\rho_d(\mathbf{k})(s+s',x+x',y+y'+sx'_{\uparrow},t+t'+\frac{x y'-y x' +sx x'_{\uparrow}}{2})\\
     &=(s+s',x+x',y+y'+sx'_{\uparrow}+k_0 (x+x')_{\uparrow},t+t'+\frac{x y'-y x'+s x x'_{\uparrow}}{2}+k (x+x'))\\
     &=\rho_d(\mathbf{k})((s,x,y,t))\rho_d(\mathbf{k})((s',x',y',t')),
 \end{split}
 \end{gather*}
 where we use $xx'_{\uparrow}=x'x_{\uparrow}$,
 and
 \begin{gather*}
     \begin{split}
         \rho_d(\mathbf{k}\mathbf{k'})((s,x,y,t))
         &=(s,x,y+(k_0+k'_0)x,t+(k+k')x)\\
         &=\rho_d(\mathbf{k})(\rho_d(\mathbf{k'})((s,x,y,t))).
     \end{split}
 \end{gather*}
 
 In order to prove that $K_d$ is a subgroup of $\Aut(N_d)$, it suffices to prove that $\rho_d$ is injective.
 If $\rho_d(\mathbf{k})((s,x,y,t))=(s,x,y,t)$ for all $(s,x,y,t)\in N_{d}$, then we have 
 $\rho_d(\mathbf{k})((s,x,y,t))=(s,x,y+k_0 x_{\uparrow},t+kx)=(s,x,y,t)$, which implies that $k_i=0$ for any $i=0,\cdots, d$.
 This completes the proof.
\end{proof}

The following is the main theorem, which we prove in the rest of this paper.

\begin{theorem}\label{maintheorem}
    For each $d\ge 1$, the pair $(K_d\ltimes N_d,K_d)$ is a generalized Gelfand pair.
\end{theorem}

\begin{remark}
    If $d=1$, then the pair $(K_1\ltimes N_1,K_1)$ recovers the example of Gallo and Saal \cite{gallo2020generalized}.
\end{remark}

\subsection{Kirillov's theory}
We briefly review Kirillov's theory \cite{kirillov2025lectures} which establishes a bijection between the set $\hat{N_d}$ of equivalence classes of irreducible unitary representations of $N_d$ and the set of coadjoint orbits for the real dual space $\mathfrak{n}_d^*$ of $\mathfrak{n}_d$.

The adjoint action of $N_{d}$ on $\mathfrak{n}_{d}$
\begin{gather*}
    \Ad: N_{d}\to \Aut(\mathfrak{n}_{d})
\end{gather*}
is given for $g=(s,x,y,t)\in N_{d}$, $n=(S,X,Y,T)\in \mathfrak{n}_{d}$ by
\begin{gather*}
    \Ad(g)(n)=(S,X,Y+s X_{\uparrow}-S x_{\uparrow},T+xY-yX+sx X_{\uparrow}-\frac{S}{2}x x_{\uparrow}),
\end{gather*}
and thus the coadjoint action of $N_{d}$ on $\mathfrak{n}_{d}^*$
\begin{gather*}
    \Ad^*: N_{d}\to \GL(\mathfrak{n}_{d}^*)
\end{gather*}
is given for $g=(s,x,y,t)\in N_{d}$, $\Lambda=(\alpha,\mu,\nu,\lambda)\in \mathfrak{n}_{d}^*$ by
\begin{gather}\label{coadjointaction}
    \Ad^*(g)(\Lambda)
    =(\alpha+\nu x_{\uparrow}-\frac{\lambda}{2}x x_{\uparrow}, \mu-s\nu_{\uparrow}+\lambda y, \nu-\lambda x,\lambda).
\end{gather}
The coadjoint orbit $\calO_{\Lambda}$ for $\Lambda\in \mathfrak{n}_{d}^*$ is defined by 
\begin{gather*}
    \calO_{\Lambda}=\{\Ad^*(g)(\Lambda)\in\mathfrak{n}_{d}^*\mid g\in N_{d}\}\subset \mathfrak{n}_{d}^*.
\end{gather*}

For $\Lambda\in \mathfrak{n}_d^*$, we have a skew-symmetric form
\begin{gather*}
B_{\Lambda}: \mathfrak{n}_d\times \mathfrak{n}_d\to \R
\end{gather*}
defined for $n,n'\in \mathfrak{n}_d$ by $B_{\Lambda}(n,n')=\Lambda([n,n'])$.
Let $\mathfrak{M}_{\Lambda}\subset \mathfrak{n}_d$ be a maximal isotropic subspace associated to $B_\Lambda$, where a subspace is isotropic if the restriction of $B_{\Lambda}$ vanishes.
Let $M_{\Lambda}=\exp(\mathfrak{M}_{\Lambda})$, 
and let $\chi_{\Lambda}: M_{\Lambda}\to \C$ be the character defined for $n\in \mathfrak{M}_{\Lambda}$ by
$\chi_{\Lambda}(\exp(n))=e^{i \Lambda(n)}$.
Then the irreducible unitary representation of $N_{d}$ that corresponds to the coadjoint orbit $\mathcal{O}_\Lambda$ is
\begin{gather*}
    (\rho_{\Lambda}, H_{\Lambda})=\operatorname{Ind}^{N_{d}}_{M_{\Lambda}} (\chi_{\Lambda}),
\end{gather*}
where $\rho_{\Lambda}$ is defined by 
\begin{gather*}
    (\rho_{\Lambda}(n))(f)(n')=f(n^{-1}n'), \quad n,n'\in N_{d}, f\in H_\Lambda,
\end{gather*}
and where $H_{\Lambda}$ is the completion of 
\begin{gather*}
   \{f\in C_{c}(N_{d})\mid f(nm)=\chi_{\Lambda}(m^{-1})f(n),\; m\in M_{\Lambda}, n\in N_{d}\} 
\end{gather*}
with respect to the inner product $\langle, \rangle: C_{c}(N_{d})\times C_{c}(N_{d})\to \C$ defined by $\langle f,g \rangle=\int_{N_{d}/M_{\Lambda}} f(u)\overline{g(u)}du$.

\subsection{A classification of coadjoint orbits for $\mathfrak{n}_d^*$}

In \cite{ratcliff1985symbols}, Ratcliff studied the \emph{generic} coadjoint orbits of any simply-connected Lie group whose Lie algebra is a $3$-step nilpotent Lie algebra with $1$-dimensional center. 
Here, we explicitly describe the coadjoint orbits $\calO_{\Lambda}$ with $\Lambda=(\alpha,\mu,\nu,\lambda)\in \mathfrak{n}_{d}^*$.

Note that the relation \eqref{coadjointaction} implies that the last coordinate $\lambda$ is preserved under the coadjoint action; that is, for each $\Lambda$, all elements of the coadjoint orbit $\calO_{\Lambda}$ have the same last coordinate. 

We first consider the case of \emph{generic orbits}, i.e., $\lambda\neq 0$.
Considering the action of $g=(0,\frac{1}{\lambda}\nu, -\frac{1}{\lambda}\mu,0)$, we see that the orbit $\mathcal{O}_{\Lambda}$ includes $(\alpha+\frac{1}{2\lambda}\nu\nu_{\uparrow},0,0,\lambda)$.
Therefore, in this case, we can take $(\alpha,0,0,\lambda)$ as a representative of the coadjoint orbit, and we obtain
\begin{gather*}
    \calO_{(\alpha,0,0,\lambda)}=\{(\alpha-\frac{1}{2\lambda}\nu \nu_{\uparrow},\mu,\nu,\lambda)\mid \mu,\nu\in \R^d\}. 
\end{gather*}

We next consider the case of non-generic orbits, i.e., $\lambda=0$.
In this case, the last two coordinates $\nu$ and $\lambda$ are preserved.
If $\nu=0$, then $(\alpha,\mu,0,0)$ is preserved under the coadjoint action; that is, the orbit $\mathcal{O}_{(\alpha,\mu,0,0)}$ consists of just one element $(\alpha,\mu,0,0)$.
Otherwise, we have $\nu_i\neq 0$ for some $1\le i\le d$, and thus we see that the orbit $\mathcal{O}_{\Lambda}$ includes $(0,\mu,\nu,0)$, considering the action of $g=(0,x,0,0)$ with $x_{d+1-i}=-\frac{1}{\nu_i}\alpha$, $x_j=0$ for $j\neq d+1-i$.
Therefore, in this case, we can take $(0,\mu,\nu,0)$ as a representative of the coadjoint orbit, and we obtain
\begin{gather*}
    \calO_{(0,\mu,\nu,0)}=\{(\alpha,\mu-s \nu_{\uparrow},\nu,0)\mid \alpha,s\in \R\}.
    \end{gather*}

Therefore, we obtain the following classification of the coadjoint orbits for $\mathfrak{n}_d^*$.
\begin{lemma}
We have 
$$\{\mathcal{O}_{\Lambda}\mid \Lambda\in \mathfrak{n}_d^*\}=
\{\mathcal{O}_{(\alpha,0,0,\lambda)}\mid \lambda\neq 0\}\sqcup 
\{\mathcal{O}_{(0,\mu,\nu,0)}\mid \nu\neq 0\}\sqcup
\{\mathcal{O}_{(\alpha,\mu,0,0)}\},$$
where
\begin{gather*}
    \calO_{(\alpha,0,0,\lambda)}=\{(\alpha-\frac{1}{2\lambda}\nu \nu_{\uparrow},\mu,\nu,\lambda)\mid \mu,\nu\in \R^d\},\\
    \calO_{(0,\mu,\nu,0)}=\{(\alpha,\mu-s \nu_{\uparrow},\nu,0)\mid \alpha,s\in \R\},\quad \mathcal{O}_{(\alpha,\mu,0,0)}=\{(\alpha,\mu,0,0)\}.
\end{gather*}
\end{lemma}

\subsection{Irreducible unitary representations of $N_d$}

Let $\Lambda$ be of the form $(\alpha,0,0,\lambda)$ with $\lambda\neq 0$, $(0,\mu,\nu,0)$ with $\nu\neq 0$ or $(\alpha,\mu,0,0)$. 
Here we describe the irreducible representation $\rho_{\Lambda}$ of $N_{d}$ corresponding to the coadjoint orbit $\mathcal{O}_\Lambda$.

Since the skew-symmetric form $B_{\Lambda}: \mathfrak{n}_{d}\times \mathfrak{n}_{d}\to \R$ is 
\begin{gather*}
\begin{split}
    B_{\Lambda}((S,X,Y,T),(S',X',Y',T'))&=\Lambda([(S,X,Y,T),(S',X',Y',T')])\\
    &=\Lambda\left((0,0,S X'_{\uparrow}-S' X_{\uparrow}, X Y'-Y X')\right),
\end{split}
\end{gather*}
a maximal isotropic subspace associated to $\Lambda$ can be taken as 
\begin{gather*}
\mathfrak{M}_{(\alpha,0,0,\lambda)}=\mathfrak{M}_{(0,\mu,\nu,0)}=\{(S,0,Y,T)\mid S,T\in \R, Y\in \R^{d}\},\quad \mathfrak{M}_{(\alpha,\mu,0,0)}=\mathfrak{n}_d.
\end{gather*}
Since for $(s,x,y,t)\in N_{d}$, we have 
\begin{gather*}
    (s,x,y,t)=(0,x,0,0)\left(s,0,y,t-\frac{x y}{2}\right),
\end{gather*}
and since $(s,0,y,t-\frac{x y}{2})$ is an element of the above maximal isotropic subspace in any case, we can identify the representation space $H_{\Lambda}$ with $L^2(\R^{d})$ via the identification $(0,x,0,0)\mapsto x\in \R^d$.

Let $\rho_{\alpha,\lambda}$ (resp. $\rho_{\mu,\nu}$, $\rho_{\alpha,\mu}$) denote the representation $\rho_{(\alpha,0,0,\lambda)}$ (resp. $\rho_{(0,\mu,\nu,0)}$, $\rho_{(\alpha,\mu,0,0)}$).

\begin{lemma}\label{computationofrho}
 Let $f\in H_{\Lambda}\cong L^2(\R^d), u\in \R^d$.
 We have for $\Lambda=(\alpha,0,0,\lambda)$ with $\lambda\neq 0$,
    \begin{gather*}
        \begin{split}
            ((\rho_{\alpha,\lambda}(s,0,0,0))(f))(u)&=e^{i (\alpha s- \frac{1}{2}\lambda s u u_{\uparrow})} f(u),\\
            ((\rho_{\alpha,\lambda}(0,x,0,0))(f))(u)&=f(u-x),\\
            ((\rho_{\alpha,\lambda}(0,0,y,0))(f))(u)&=e^{-i \lambda u y} f(u),\\
            ((\rho_{\alpha,\lambda}(0,0,0,t))(f))(u)&=e^{i \lambda t} f(u)
        \end{split}
    \end{gather*}
    and for $\Lambda=(0,\mu,\nu,0)$ with $\nu\neq 0$,
    \begin{gather*}
        \begin{split}
            ((\rho_{\mu,\nu}(s,0,0,0))(f))(u)&=e^{i s\nu u_{\uparrow}} f(u),\\
            ((\rho_{\mu,\nu}(0,x,0,0))(f))(u)&=f(u-x),\\
            ((\rho_{\mu,\nu}(0,0,y,0))(f))(u)&=e^{i \nu y} f(u),\\
            ((\rho_{\mu,\nu}(0,0,0,t))(f))(u)&=f(u)
        \end{split}
    \end{gather*}
    and for $\Lambda=(\alpha,\mu,0,0)$,
     \begin{gather*}
        \begin{split}
            ((\rho_{\alpha,\mu}(s,0,0,0))(f))(u)&=e^{i \alpha s} f(u),\\
            ((\rho_{\alpha,\mu}(0,x,0,0))(f))(u)&=f(u-x),\\
            ((\rho_{\alpha,\mu}(0,0,y,0))(f))(u)&=f(u),\\
            ((\rho_{\alpha,\mu}(0,0,0,t))(f))(u)&=f(u).
        \end{split}
    \end{gather*}
\end{lemma}

\begin{proof}
    We will check the computation of $((\rho_{\alpha,\lambda}(s,0,0,0))(f))(u)$. One can check the other cases in a similar way.
    We have
    \begin{gather*}
    \begin{split}
     ((\rho_{\alpha,\lambda}(s,0,0,0))(f))(u)
     &= f((s,0,0,0)^{-1} (0,u,0,0))\\
     &= f((-s,0,0,0) (0,u,0,0))\\
     &=f((-s,u,-su_{\uparrow},0))\\
     &=f((0,u,0,0) (-s,0,-su_{\uparrow},\frac{1}{2}s u u_{\uparrow}))\\
     &=\chi_{(\alpha,0,0,\lambda)}((-s,0,-su_{\uparrow},\frac{1}{2}s u u_{\uparrow})^{-1})f(u)\\
     &=\chi_{(\alpha,0,0,\lambda)}((s,0,su_{\uparrow},-\frac{1}{2}s u u_{\uparrow}))f(u)\\
     &=e^{i (\alpha s- \frac{1}{2}\lambda s u u_{\uparrow})} f(u).
    \end{split}
    \end{gather*}
\end{proof}

\subsection{Proof of Theorem \ref{maintheorem}}

By Proposition \ref{GalloSaalcriterion}, in order to prove Theorem \ref{maintheorem}, we need to prove that for each $d\ge 1$, for any $\Lambda\in\mathfrak{n}_d^*$, there exists a unitary representation $\omega_{\Lambda}$ of $K_d$ 
such that 
\begin{itemize}
    \item[(1)] $\rho_{\Lambda}(\mathbf{k}\cdot n)=\omega_{\Lambda}(\mathbf{k})\rho_{\Lambda}(n)\omega_{\Lambda}(\mathbf{k}^{-1})$ for any $n\in N_d$, $\mathbf{k}\in K_d$,
    \item[(2)] $\omega_{\Lambda}$ is multiplicity free.
\end{itemize}

For $\Lambda=(\alpha,0,0,\lambda)$ with $\lambda\neq 0$, we define $\omega_{\Lambda}=\omega_{\alpha,\lambda}$ by
\begin{gather*}
     (\omega_{\alpha,\lambda}(\mathbf{k})(f))(u)=e^{i\lambda (-\frac{1}{2}k_0 u u_{\uparrow}+ ku)}f(u)
\end{gather*}
for $\mathbf{k}=(k_0,k)\in K_d\cong \R^{d+1}$ with $k_0\in \R$, $k=(k_1,\dots,k_d)\in \R^d$, $f\in H_{\Lambda}\cong L^2(\R^d)$, $u\in \R^d$.
Then it is easy to see that $\omega_{\alpha,\lambda}$ is a unitary representation of $K_d$.
We will check condition (1). Since $\omega_{\alpha,\lambda}$ and $\rho_{\alpha,\lambda}$ are representations, it suffices to check that $\rho_{\alpha,\lambda}(\mathbf{k}\cdot n)\omega_{\alpha,\lambda}(\mathbf{k})=\omega_{\alpha,\lambda}(\mathbf{k})\rho_{\alpha,\lambda}(n)$ for $\mathbf{k}=(0,\cdots, k_j,\cdots, 0)$ for some $0\le j\le d$ and $n=(s,0,0,0),(0,x,0,0),(0,0,y,0),(0,0,0,t)$. 
The cases of $n=(s,0,0,0),(0,0,y,0),(0,0,0,t)$ are obvious since we have $\mathbf{k}\cdot n=n$ and since $\rho_{\alpha,\lambda}(n)(f)(u)$ and $\omega_{\alpha,\lambda}(\mathbf{k})(f)(u)$ are scalar multiples of $f(u)$.
For $n=(0,x,0,0)$, we have
\begin{gather*}
\begin{split}
    \rho_{\alpha,\lambda}(k_0\cdot (0,x,0,0))\omega_{\alpha,\lambda}(k_0)(f)(u)
    &=\rho_{\alpha,\lambda}(0,x,k_0 x_{\uparrow},0)\omega_{\alpha,\lambda}(k_0)(f)(u)\\
    &=e^{i(-\frac{1}{2}\lambda k_0uu_{\uparrow})}f(u-x)\\
    &=\omega_{\alpha,\lambda}(k_0)\rho_{\alpha,\lambda}((0,x,0,0))(f)(u)
\end{split}
\end{gather*}
by Lemma \ref{computationofrho} since we have 
\begin{gather*}
    (0,x,k_{0}x_{\uparrow},0)=(0,x,0,0)(0,0,k_{0}x_{\uparrow},0)(0,0,0,-\frac{k_{0}x x_{\uparrow}}{2}).
\end{gather*}
We also have 
\begin{gather*}
    \begin{split}
         \rho_{\alpha,\lambda}(k_j\cdot (0,x,0,0))\omega_{\alpha,\lambda}(k_j)(f)(u)
         &=\rho_{\alpha,\lambda}((0,x,0,k_{j}x_{j}))\omega_{\alpha,\lambda}(k_j)(f)(u)\\
         &=\rho_{\alpha,\lambda}((0,x,0,0)(0,0,0,k_{j}x_{j}))\omega_{\alpha,\lambda}(k_j)(f)(u)\\
         &=e^{i\lambda k_j u_j} f(u-x)\\
         &=\omega_{\alpha,\lambda}(k_j)\rho_{\alpha,\lambda}((0,x,0,0))(f)(u)
    \end{split}
\end{gather*}
for $1\le j\le d$.
Therefore, condition (1) holds.
Condition (2) also holds since we have the following decomposition of the representation $\omega_{\alpha,\lambda}$ on $L^2(\R^d)$ 
\begin{gather*}
    L^2(\R^{d})=\int_{\R^{d}} \chi_{-\lambda \frac{u u_{\uparrow}}{2},\lambda u_1,\dots,\lambda u_{d}}\; du,\quad 
    \chi_{-\lambda \frac{u u_{\uparrow}}{2},\lambda u_1,\dots,\lambda u_{d}}(\mathbf{k})=e^{i\lambda (-\frac{1}{2}k_0 uu_{\uparrow}+ ku)}.
\end{gather*}

For $\Lambda=(0,\mu,\nu,0)$ with $\nu\neq 0$ or $\Lambda=(\alpha,\mu,0,0)$, we define $\omega_{\Lambda}=\omega_{\mu,\nu}, \omega_{\alpha,\mu}$ by
\begin{gather*}
    (\omega_{\mu,\nu}(\mathbf{k})(f))(u)=e^{i k_0 \nu u_{\uparrow}}f(u),\quad  (\omega_{\alpha,\mu}(\mathbf{k})(f))(u)= f(u).
\end{gather*}
Then by a similar argument, we can check that $\omega_{\mu,\nu}$ and $\omega_{\alpha,\mu}$ satisfy conditions (1) and (2); in fact, we have the following decompositions 
\begin{gather*}
    L^2(\R^{d})=\int_{\R^{d}} \chi_{\nu u_{\uparrow},0,\dots,0}\; du,\quad \chi_{\nu u_{\uparrow},0,\dots,0}(\mathbf{k})=e^{i\nu u_{\uparrow} k_0},\\
    L^2(\R^{d})=\int_{\R^{d}}\; du.
\end{gather*}
This completes the proof of Theorem \ref{maintheorem}.

\section*{Acknowledgements}
We would like to thank Professor Linda V. Saal for helpful comments and for letting us know about their recent works on generalized Gelfand pairs.

\bibliographystyle{plain}
\bibliography{main}

\end{document}